\documentclass[12pt]{amsart}
\usepackage{amsmath,amsfonts}
\usepackage{amssymb}
\usepackage{amscd}
\usepackage{amsthm}
\usepackage{yhmath}
\usepackage{subfigure}
\usepackage[all]{xy}
\usepackage{color}

\numberwithin{equation}{section}

\newtheorem{theorem}[equation]{Theorem}

\newtheorem{lemma}[equation]{Lemma}

\newtheorem{corollary}[equation]{Corollary}

\theoremstyle{remark}
\newtheorem{remark}[equation]{Remark}

\theoremstyle{definition}
\newtheorem{definition}[equation]{Definition}

\newtheorem{example}[equation]{Example}

\def\XXint#1#2#3{{\setbox0=\hbox{$#1{#2#3}{\int}$}
	\vcenter{\hbox{$#2#3$}}\kern-.5\wd0}}

\newcommand{\N}{\mathbb N}

\newcommand{\R}{\mathbb R}

\newcommand{\Z}{\mathbb Z}

\newcommand{\G}{{\mathbb G}}
\newcommand{\g}{\mathfrak{g}}
\newcommand{\z}{\mathfrak{z}}

\newcommand{\Lie}{\operatorname{Lie}}

\begin{document}

\title[Ultrarigid tangents of sub-Riemannian nilpotent groups]{Ultrarigid tangents of sub-Riemannian nilpotent groups
} %\\Tangentes ultrarigides des groupes nilpotents sous-riemanniennes}

%\date{\today}
\author{Enrico Le Donne}
\address{
Department of Mathematics and Statistics, University of Jyv\"askyl\"a, 40014 Jyv\"askyl\"a, Finland}
\email{enrico.ledonne@jyu.fi}

%\address{Enrico Le Donne , ETH Z\"urich \\ 8092 Z\"urich, Switzerland}
%\email{enrico.ledonne@math.ethz.ch   }

\author{Alessandro Ottazzi}
\address{
CIRM Fondazione Bruno Kessler, Via Sommarive 14, 38123 Trento, Italy}
\email{ottazzi@fbk.eu}

\author{Ben Warhurst}

\address{Faculty of Mathematics Infomatics and Mechanics\\ University of Warsaw, Poland}
\email{benwarhurst@mimuw.edu.pl}

\date{December 22, 2013}

\thanks{During the initial stages of this project, E.L.D. was supported by ETH Z\"urich, while
B.W. was supported by Universit\`a di Milano Bicocca and Universit\"at Bern.}

 
\begin{abstract}
We show that the tangent cone at the identity is not a complete quasiconformal invariant for sub-Riemannian nilpotent groups.
Namely, we show that there exists a nilpotent Lie group equipped with left invariant
sub-Riemannian metric that is not locally quasiconformally equivalent to
its tangent cone at the identity.
In particular, such spaces are not locally bi-Lipschitz homeomorphic.
The result is based on the study of Carnot groups that are rigid in the sense that their only 
 quasiconformal maps are the translations and the dilations.
\end{abstract}

\maketitle
\tableofcontents

\section{Overture}%{Introduction}
\subsection{Overview of the results}
By means of a result  \cite{Margulis-Mostow} of Margulis and Mostow,   if two equiregular sub-Riemannian manifolds are quasiconformally equivalent, then 
almost everywhere they have isomorphic tangent cones.
   In particular, the tangent cone is a quasiconformal invariant.
 % they have the same tangent cone at every point.
 Their work extends   a result \cite{pansu1} of   Pansu, for which two Carnot groups are quasiconformally equivalent only if they are isomorphic.

The main goal of this paper is to show that the converse of the theorem of Margulis and Mostow fails in a strong sense. We show  the following statement.
\begin{theorem}\label{main}
There exists a nilpotent Lie group equipped with left invariant sub-Riemannian metric that is not locally quasiconformally equivalent to its tangent cone.
\end{theorem}
Note that the theorem above holds in particular for locally bi-Lipschitz maps.
We recall that the tangent cone of an equiregular sub-Riemannian manifold is a Carnot group, and it coincides with the nilpotentisation of its sub-Riemannian structure~\cite{Mitchell}. In order to establish the result, we shall study groups with the property that whenever they are quasiconformally equivalent to some other group,  they are in fact  isomorphic to it, see Theorem \ref{main2}. With this purpose in mind, we consider  %This is done by selecting
   Carnot groups whose quasiconformal maps are only translations and dilations. We shall refer to groups with this property as {\em ultrarigid groups}.
In order to show that some group is ultrarigid, we  prove the following algebraic characterization.
\begin{theorem}\label{ultra-rigidity}\label{ultralip}
Let  $\G$ be a   Carnot group. Then the following are equivalent:
\begin{itemize}
\item[{[\ref{ultra-rigidity}.1]}]   For any $U\subset \G $ open and any quasiconformal embedding $f: U\to \G $, one has that  $f$ is the restriction of the
composition of
a  left translation  and a dilation;
\item[{[\ref{ultra-rigidity}.2]}]  %The group ${\rm Aut}_0(\g)$ of 
Every strata preserving automorphism of Lie$(\G)$ is a dilation.
 %coincides with the dilation group.
\end{itemize}
\end{theorem}
The class of groups defined by [\ref{ultra-rigidity}.2] was considered by Pansu in~\cite{pansu1}. He showed that there exist infinitely many $2$-step Carnot groups with this property, although his proof does not provide explicit examples. 
 We exhibit two examples of groups satisfying [\ref{ultra-rigidity}.2].
The first one is a $2$-step stratified nilpotent Lie group, whereas the latter has step $3$ and it is not in the class studied by Pansu. Finally, we point out     that,  in the case of $2$-step Carnot groups,  the nontrivial implication $[\ref{ultra-rigidity}.2]\Rightarrow [\ref{ultra-rigidity}.1]$ of Theorem~\ref{ultra-rigidity} was proved by Capogna and Cowling, see~\cite{Capogna-Cowling}.

\subsection{State of the art}
Given a metric space $(X,d)$ and a base point $x\in X$, one can consider the {\em blow-down spaces } and the {\em blow-up spaces} of $X$ at $x$.
Namely, a pointed metric space $((Z,\rho), z)$ is a  blow-up (resp. a blow-down) of $X$ at $x$ if there exists a sequence of positive real numbers $\lambda_j$ with $\lambda_j\to \infty$ (resp. $\lambda_j\to 0$), as $ j\to \infty$, such that
$((X,\lambda_j d), x)$ Gromov-Hausdorff converges to $((Z,\rho), z)$.
Such blow-down spaces and  blow-up spaces are not unique and do not always exist.
Whenever the limit exists for any sequence $\lambda_j\to \infty$ (resp. $\lambda_j\to 0$) and does not depend on the sequence,  the  blow-up (resp.  blow-down) space is called the {\em tangent cone} (resp. {\em asymptotic cone}).
In many situations a map between metric spaces induces a map between blow-down  or   blow-up spaces.
 A key fact is that % It is interesting/useful that 
  the induced map has often more geometric structure than the initial map.

We recall two examples of blow-down and   blow-up spaces, which are well known in sub-Riemannian geometry and in Geometric Group Theory, respectively.
Let $M$, $M'$ be two manifolds endowed with some   sub-Riemannian distances induced by equiregular horizontal distributions.
In such a setting, the blow-up spaces do not depend on the scaling sequences and are stratified nilpotent Lie groups, see \cite{Mitchell}.
Let $f:M\to M'$ be a quasiconformal homeomorphism.
According to    \cite{Margulis-Mostow}, for almost every $p\in M$, the map $f$ blows up at $p$ to a strata preserving  group isomorphism between the blow-up space at $p$ and the one at $f(p)$.
Regarding the large scale geometry of groups,
let $\Gamma$ be a finitely generated nilpotent group. Endow $\Gamma$ with any word distance induced by a finite generating set. The unfamiliar reader might just think that  $\Gamma$ is a connected, simply connected nilpotent Lie group endowed with a Riemannian left invariant distance.
By \cite{Pansu-croissance}, the blow-down space of $\Gamma$ is unique and is a stratified nilpotent Lie group endowed with a left invariant Carnot-Carath\'eodory distance, induced by a norm on the first stratum.
Likewise the general framework, any quasi-isometry blows down to a bi-Lipschitz homeomorphism of the blow-down spaces. Consequently, such blow-down spaces are isomorphic.

 Once a map is given at the blow-up or at the blow-down level,
it is then natural to ask if we can integrate back to a map between the initial spaces.
Namely, if we are given two sub-Riemannian manifolds with isomorphic blow-up spaces at a point (resp. two finitely generated nilpotent groups with isomorphic blow-down spaces), to what extent we may conclude that the two manifolds are quasiconformally equivalent (resp. the nilpotent groups are quasi-isometric)?

The fact that the blow-down space is not a complete quasi-isometric invariants was proved by Shalom~\cite{Shalom}, using group cohomology. Namely, he shows that quasi-isometric nilpotent groups have same Betti numbers. Then he exhibits an example due to Benoist of two nilpotent groups with same blow-down space and different Betti numbers. We summarize the result of Shalom as the following statement.
\begin{theorem}[{Shalom \cite[page 152]{Shalom}}]
There exist  two  finitely generated nilpotent   groups $\Gamma$ and $\Lambda$ that have the same blow-down space, but they are not   quasi-isometric equivalent.
\end{theorem}

Although blow-down spaces capture the asymptotic geometry, Shalom's result shows that they do not capture the whole large scale geometry.
Similarly, for general sub-Riemannian manifolds, blow-up spaces capture only the infinitesimal geometry and not the local geometry.
To see this, one can consider an example\footnote{The existence of sub-Riemannian manifolds  whose blow-up space varies continuously was noticed by Pansu. An explicit $11$-dimensional example has being given by Varchenko in \cite{Varchenko}.}
 of a sub-Riemannian manifold $M$ whose blow-up space is not constant on a full measure set. Indeed, fix $p\in M$ and let $\G$ be the blow-up of $M$ at $p$. We claim that no neighborhood of $p$ is quasiconformally equivalent to an open set in $\G$.
Since $\G$ is a cone, it is isometric to its blow-up space. Hence, the blow-up of $M$ at $p$ is isomorphic to the blow-up of $\G$ at the identity. Assume by contradiction that there exists  a quasiconformal embedding $f:U\subset M \to \G$. Then by   \cite{Margulis-Mostow} almost every blow-up space in $U$ needs to be isomorphic to $\G$. Since this is not the case, such a quasiconformal  map does not exist.

We conclude that a necessary condition for a sub-Riemannian manifold $M$ to be quasiconformally equivalent to a Carnot group $\G$ is that the blow-up space of $M$  is $\G$ at almost every point.
It is then natural to ask what happens when the manifold has the same   blow-up space at   every point.
Here the work of Pansu, Margulis, and Mostow fails to give an answer. 
One needs to find a different strategy.

A natural example of a manifold  with constant blow-up spaces is provided by a Lie group  $G$ endowed with a left invariant sub-Riemannian distance.
In this case, the isometry group of $G$ acts transitively on $G$.
Therefore the blow-up space is the same at every point   of $G$. 
In this article, we provide a nilpotent Lie group of dimension $16$ that is not locally quasiconformal (and hence not locally bi-Lipschitz) equivalent to its blow-up. In particular, we have the following consequence.
\begin{corollary}\label{corollario}
There exist  two sub-Riemannian nilpotent Lie groups $H$ and $G$, that have the same blow-up space at every point, but they are not (locally) quasiconformally equivalent.
\end{corollary}

We conclude our survey section by recalling some positive results: blow-up spaces  of Riemannian manifolds   and contact $3$-manifolds
do capture the local geometry.
Indeed, every point $p$ in a Riemannian $n$-manifold $M$ has a neighborhood that is bi-Lipschitz equivalent to an open set in $\R^n$, which is the blow-up   of $M$ at $p$.
The same phenomenon appears for  contact $3$-manifolds. By  Darboux's Theorem, every point in a contact manifold  has a neighborhood that is contactomorphic to an open set of the standard contact structure.
We can   see this as a metric statement. Indeed, every contact manifold can be endowed with a sub-Riemannian structure, which is unique up to bi-Lipschitz equivalence.
Now,  Darboux's Theorem implies that
every point $p$ in a sub-Riemannian $3$-manifold $M$ has a neighborhood that is bi-Lipschitz equivalent to an open set in the sub-Riemannian standard contact structure. 
 %Such a  sub-Riemannian standard contact structure
The latter is   the sub-Riemannian Heisenberg group.
Since the Heisenberg group is dilation invariant, we also have that
the sub-Riemannian Heisenberg group      is the blow-up at any point of 
any sub-Riemannian $3$-manifold.
We can therefore conclude that every nilpotent Lie group $G$ of dimension $3$ has the property that, when it is endowed with a left invariant sub-Riemannian metric,   any element of $G$ has a neighborhood that is bi-Lipschitz homeomorphic to an open set in the blow-up space of $G$.

We remark that in the setting of Riemannian groups or of $3$-dimensional contact groups, the blow-up spaces may not preserve the algebraic structure of the original space. Examples in the Riemannian setting are easy to find, because there are diffeomorphic Lie groups that are not isomorphic. On the other hand, the sub-Riemannian roto-translation group is not isomorphic to its blow-up space, which is the Heisenberg group.

\subsection{Structure of the paper}
%{\color{red} TO BE DONE}
The article is organized as follows. In Section~\ref{notation} we fix notation and  state the results of the literature
that are the building blocks of our  work. In Section~\ref{ultra} we restate Theorem~\ref{ultralip}, which characterizes ultrarigidity in purely Lie theoretic terms. Then we give two examples of  ultrarigid Carnot groups. In Section~\ref{counter} we establish our main results. To begin we prove Theorem~\ref{main2}, which is a rigidity type theorem for  sub-Riemannian nilpotent Lie groups with ultrarigid tangent cone.  Secondly, we exhibit two example of a sub-Riemannian nilpotent Lie group with ultrarigid tangent cone. This together with Theorem~\ref{main2} imply Theorem~\ref{main} and Corollary~\ref{corollario}. In Section~\ref{section-proof-equivalence} we recall the definition of Tanaka prolongation of a stratified nilpotent Lie algebra and state a result of Tanaka. We then use this to
prove Theorem~\ref{ultralip}.

\section{Notation and preliminaries}\label{notation}
\subsection{Carnot Groups} \label{CarnotGroups}
Let $\G$ be a stratified nilpotent Lie group with identity $e_\G$ or $e$ if no confusion arises.
This means that its Lie algebra $\g$ admits an $s$-step stratification
$$\g= V_{1}\oplus \cdots\oplus V_{s}, $$
where $[V_{j}, V_{1}] =V_{j+1}$, for $1\leq j \leq s$, 
and with  $V_{s}\neq \{0\}$ and $V_{s+1}=\{0\}$. To avoid degeneracies, we assume $\g$ to have at least dimension two, which is reasonable for our purposes.

Given a point $p \in \G$ we denote by $\tau_p$ the left translation by $p$. 
An element $X$ in the Lie algebra $\g$ can be considered as a tangent vector at the identity.
Such a vector induces the left invariant vector field  given by $(\tau_p)_*|_e(X)$ at a point $p \in\G$.  This vector field will still be denoted by $X$, unless confusion might arise.  
%The left invariant vector field $\tilde X \in \mathfrak{X}(\G)$, corresponding to $X \in \g$, is given by $\tilde X(g)=(\tau_g)_*|_e(X)$. 
 The set of all left invariant vector fields with the bracket operation is isomorphic to $\g$ and it inherits the   stratification of $\g$. The sub-bundle   $\mathcal{H} \subseteq TG$ where $\mathcal{H}_p= (\tau_p)_*|_e(V_{1})$ is called the {\it horizontal distribution}. A scalar product $\langle \, ,\, \rangle $ on $V_{1}$ defines a left invariant scalar product on each $\mathcal{H}_p$ by setting \begin{align}
\langle v, w \rangle_p=\langle (\tau_{p^{-1}})_*|_p(v),(\tau_{p^{-1}})_*|_p(w) \rangle  \label{scalarprod}
\end{align} for all $v,w \in \mathcal{H}_p$. The left invariant scalar product gives rise to a left invariant sub-Riemannian metric $d$ on $\G$, the definition of which we shall give later in the more general setting of sub-Riemannian manifolds. We call $(\G,d)$ a {\it Carnot group}, which we simply denote by $\G$ if no ambiguity arises.

%Fix a basis of  $\mathfrak{g}={\rm span}\{X_{-i,\alpha}: i=1,\dots,s,\,\,\alpha=1,\dots,d_i\}$ in such a way that for every $i=1,\dots,s$, the set $\{X_{-i,\alpha}: \alpha=1,\dots,d_i\}$ is a basis of $\g_{-i}$. The exponential coordinate $x$ is defined according to $$g=\exp(\sum_{i,\alpha}x_{-i,\alpha}(g)X_{-i,\alpha}),$$ and we denote by $ X^l_{-i,\alpha}$ the left invariant vector field such that $X^l_{-i,\alpha}(g)=(\tau_g)_*|_e(X_{-i,\alpha})$. The corresponding dual forms are $\theta_{-i,\alpha}^l(g)=dx_{-i,\alpha}\circ (\tau_g^{-1})_*|_g$.

We denote by ${\rm Aut}_0(\g)$ the Lie group of strata preserving automorphisms of $\g$. The Lie algebra of ${\rm Aut}_0(\g)$ is the space of strata preserving derivations of $\g$, which we denote by $\g_0$.  In general, for any stratified nilpotent Lie algebra, there are distinguished elements  of ${\rm Aut}_0(\g)$, which are called  {\em dilations} (or better  {\em algebra-dilations}). For each $\lambda \in \R$, the dilation $\delta_\lambda$ is defined linearly by setting $\delta_\lambda(X):=\lambda^jX$, for every $X\in V_{j}$  and every $j=1,\dots,s$.  The subset $\{\delta_\lambda \, |\,  \lambda\in \R\setminus\{0\} \}$ is called the {\em algebra-dilation group}. The set of dilations with positive factor constitutes a one parameter subgroup of ${\rm Aut}_0(\g)$, whose Lie algebra is generated by      the derivation $D \in \g_0$ defined by $D(X):=jX$, for every $X\in V_{j}$ and every $j=1,\dots,s$. In particular, for every $t \in \R$, we have %set $\lambda_t=e^t$,  then  
$\delta_{e^t}=e^{tD}$.  
Since any (algebra-)dilation $\delta_\lambda$ is an algebra homomorphism and the Lie group $\G$ is simply connected, the dilation induces a unique homomorphism on the group, which we still 
 denote by $\delta_\lambda$
and call {\em group-dilation} or {\em dilation on the
group level}. 
Since  $\G$ is nilpotent and simply connected, the exponential map is a diffeomorphism. Thus,   group-dilations   $\delta_\lambda$ can be defined as
 $\delta_\lambda(p)=\exp \circ \, \delta_\lambda \circ \exp^{-1}(p)$, for all $\lambda\in \R$ and  $p \in \G$. 

%is defined through the exponential, i.e.,  one has $\delta_\lambda(g)=\exp \circ \, \delta_\lambda \circ \exp^{-1}(g)$. 
%The group-dilation of a group element $g \in \G$ by $\lambda\in \R\setminus\{0\} $ is denoted by $\delta_\lambda(g)$ and is defined through the exponential, i.e.,  one has $\delta_\lambda(g)=\exp \circ \, \delta_\lambda \circ \exp^{-1}(g)$. 

 The group generated by the left translations and the group-dilations will play an important role in our considerations. We refer to this group as the {\it translation-dilation group}, and  denote it by 
 $${\rm TD}(\G):=  \left\{\tau_p\circ \delta_\lambda\;:\;p\in \G, \lambda \in \R\setminus\{0\}\right\} .$$
%The Lie group TD$(\G)$ is in fact a semidirect product $\G  \rtimes \R^+$.
%By left-translation, $\g_{-1}$ generates a sub-bundle of the tangent bundle of $\G$, which we call the {\it horizontal bundle}. A diffeomorphism $f$ from an open subset $\Omega$ of $\G$ is a {\it contact map} if its differential preserves the horizontal bundle. By definition, left translations and anisotropic scalings are contact maps.

%{\color{blue} The Haar measure on $\G$, obtained by pushing forward the Lebesgue measure on $\g$, is invariant under left translation. In exponential coordinates, this is simply the Lebesgue measure on $\R^{{\rm dim} \g}$. If we denote by $|E|_{\G}$ the measure of a set $E \subset \G$, then for $\lambda >0$ we have $|\delta_\lambda(E)|_{\G}=\lambda^Q |E|_{\G}$ where $Q=\sum_i i \, {\rm dim}\ \g_i$. It follows that $Q$ is the Hausdorff dimension of $(\G, d)$. Furthermore, the $Q$-dimensional Hausdorff measure $\mathcal{H}_{\G}^Q$ is also left inavariant and thus a Haar measure. It follows from the Haar theorem that there exists a positive constant $C$ such that $\mathcal{H}_{\G}^Q(E)=C|E|_{\G}$ for all Borel subsets of $\G$, that is, subsets in the $\sigma$-algebra generated by all compact subsets of $\G$. }
  
\subsection{Sub-Riemannian manifolds}

Throughout the paper, we shall write smooth when referring to $C^\infty$ functions, maps or vector fields.
A sub-Riemannian,  or Carnot-Carath\'eodory manifold, is a triple $(M,\mathcal{H},{\rm g} )$, where $M$ is a differentiable manifold, $\mathcal{H}$ is a bracket generating tangent sub-bundle of $M$, and ${\rm g}$ is a smooth section of the positive definite quadratic forms on $\mathcal{H}$. 

Let $m$ be $\dim  \mathcal{H}_p$, which is independent on $p\in M$.
Recall that being bracket generating means that, for every $p \in M$, there exists vector fields $  X_1,\dots,  X_{m}$ in $M$,    such that $\mathcal{H}_p={\rm span}\{  X_1(p),\dots,  X_{m}(p)\}$, and for some integer $s(p) \geq 1$,  
$$T_pM={\rm span}\, \left\{ \, [  X_{i_1},[  X_{i_2},[\dots, [  X_{i_{k-1}},  X_{i_k}] \dots ]]](p) \, :\, k=1,\dots,s(p), \qquad \right.$$
$$\qquad \qquad \qquad     \qquad \qquad \qquad \qquad \qquad \qquad \qquad \qquad
\left.i_j\in\{1,\ldots,m\}, \,j=1,\ldots,k
\, \right\}.$$ 

The bundle $\mathcal{H}$  is called the {\em horizontal distribution}.
The Carath\'eodory-Chow-Rashevsky Theorem shows that the bracket generating property implies that any two points in $M$ can be joined by a  {\em horizontal path}, i.e., an absolutely continuous path whose tangents belong to the horizontal distribution. It follows that a sub-Riemannian manifold carries a natural metric, called the  {\em sub-Riemannian} or  {\em Carnot-Carath\'eodory metric}, defined by setting
$$ d(p, q) := \inf\int_0^1 \sqrt{ {\rm g}_{ \gamma(t) }( \dot\gamma(t) , \dot\gamma(t))} dt ,$$
where the infimum is taken along all horizontal curves $\gamma: [0, 1] \to M$ such that $\gamma(0) = p$ and $\gamma(1) = q$. 

For Carnot groups, the tensor g is given by the left invariant scalar product (\ref{scalarprod}). Moreover, left translations are isometries, and $d$ is homogeneous with respect to the group-dilations, that is, $d(\delta_\lambda(p),\delta_\lambda(q))=|\lambda|d(p,q)$ for all $p,q \in \G$ and all $\lambda\in\R$.

The horizontal bundle induces a filtration of each $T_pM$ as follows:  Set $L_0=\{0\}$, let $L_1$ %=\mathfrak{X}(\mathcal{H})$ 
denote the set of all smooth sections of $\mathcal{H}$ defined on a neighbourhood of $p$, and by induction define
$L_{i+1}= L_i +[L_1,L_i]$, for $i>0$. It then follows that $$ L_0(p) \subset L_1(p) \subseteq \dots \subseteq L_{s(p)}(p)=T_pM$$ and $[L_i,L_j](p) \subseteq L_{i+j}(p)$. If $V_i(p):=L_{i}(p)/L_{i-1}(p)$, then the nilpotentisation of $T_pM$ is the vector space
 $$
 \g(p)=V_1(p)\oplus \dots \oplus V_{s(p)}(p).
 $$
  Since, for any $  X \in L_i$ and $  Y \in L_j$, one has that $$[  X+L_{i-1},  Y+L_{j-1}]=[  X,   Y] + L_{i+j-1},$$  the Lie bracket of vector fields  induces a well defined bracket on $\g(p)$ thus defining a stratified nilpotent Lie algebra of step $s(p)$. 
 %Furthermore, $\g(p)$ together with the scalar product $\langle \, , \, \rangle_p={\rm g}_p$ define a Carnot algebra.
Since  $\g(p)$ is nilpotent, by the theory of nilpotent Lie groups, there exists a unique connected, simply connected Lie group $\G_p$ whose Lie algebra is 
$\g(p)$. We might denote this group by $\exp(\g(p))$.
 %which we call $\exp(\g(p))$. 
The group $\G_p$ together with the sub-Riemannian metric $d_p$ induced by  $\langle \, , \, \rangle_p$, forms a Carnot group, which is  called the {\it tangent cone} at $p$.  
Indeed, by a theorem of Mitchell \cite{Mitchell}, the pointed metric spaces $(M,\lambda d,p)$ Gromov-Hausdorff converge, as $\lambda \to \infty$, to the pointed metric space
$(\G_p, d_p, e)$.
In other words, any blow-up space of $(M,d)$ at $p$ is isometric to the Carnot group  $\G_p$.

A sub-Riemannian manifold is called  {\em equiregular}\footnote{In Mitchell's and Margulis-Mostow's works, one finds the term `generic' instead of `equiregular'.} if the functions $p \mapsto {\rm dim}\, V_i(p)$ are constant for all $i$. Note that in this case the function $p \mapsto s(p)$ is automatically constant. The important consequence of equiregularity is that the Hausdorff dimension of $M$, with respect to $d$, is the natural number $Q=\sum_{i=1}^{s(p)} i \, {\rm dim}\, \g_i(p)$. Moreover,  on any compact subset of $M$, the $Q$-dimensional Hausdorff measure is commensurate with any Lebesgue measure, see~\cite{Mitchell}.

If the nilpotentisations are independent of $p$ and thus isomorphic to a fixed Lie algebra $\g$, then $(M,\mathcal{H})$ is said to be  {\em strongly regular} and $\g$ is called the {\em symbol algebra} of $(M,\mathcal{H})$. Clearly the tangent cones are all isomorphic to $\G=\exp(\g)$ in this case.

\subsection{Contact and quasiconformal maps}

Let $(M,\mathcal{H})$ and $(M',\mathcal{H}')$ be manifolds with horizontal distributions. Let $U\subseteq M$ be an open set. A local $C^1$ diffeomorphism $f:U \to M'$ is called a {\it contact map} if $f_*(\mathcal{H}_p)=\mathcal{H}'_{f(p)}$. In particular, left translations and dilations are contact maps of a Carnot group to itself. 

 Let $(M,d)$ and  $(M',d')$ be metric spaces and let $U\subseteq M$ be an open set.  Let $f :U \rightarrow M'$  be a (topological) embedding. For $p\in M$ and for small $t\in \R$, we define the distortion function by
$$
H_{f}(p,t):= \frac{{\rm sup}\{d'(f(p),f(q)) | d(p,q) \leq t\}}{{\rm inf}\{d'(f(p),f(q))|d(p,q)\geq t\}}.
$$
\begin{definition}\label{qcdef1}
We say that $f$ is  $K$-{\em quasiconformal} for some $K\geq 1$ if  $$\limsup_{t\rightarrow 0} H_f (p,t)\leq K,$$ for all $p\in M$. 
\end{definition}

In particular, left translations and dilations are $1$-quasiconformal maps of a Carnot group to itself. 

Quasiconformal maps between Carnot groups  $(\G,d)$ and  $(\G',d')$ are Pansu differentiable almost everywhere with respect to any Haar measure, see \cite[Th\`eor\'eme 2]{pansu1}.  
%Note that ``almost everywhere'' is with respect to Haar measure, which is in fact the Lebesgue measure on the Lie algebra pushed forward by the exponential map. 
 One such Haar measure is % An alternative equivalent measure is  
 the $Q$-dimensional Hausdorff measure, where $Q=\sum_{i=1}^s i \, {\rm dim}\, V_{i}$ is the homogeneous dimension. 

We recall that a continuous map $f:\G \to \G'$ is Pansu differentiable at $p \in \G$ if the limit
$$
 Df(p)(q)=\lim_{t\rightarrow 0^+} \delta_t^{-1}\circ \tau_{f(p)}^{-1}\circ f \circ \tau_{p}\circ \delta_t(q)
$$ 
is uniform on compact sets and equals  a homomorphism $Df(p):\G \to \G'$.
 We call  $Df(p)$ the {\em Pansu derivative} of $f$ at $p$. The {\em Pansu differential} is the Lie algebra homomorphism  $df(p):\g \to \g'$ such that $Df(p) \circ \exp=\exp \circ \, df(p)$. Note that $Df(p)$ and $df(p)$ commute with dilating and so in particular, $df(p)$ is a strata preserving  Lie algebra homomorphism.

The following results will be important for our purposes.
\begin{theorem}[L. Capogna, M. Cowling, \cite{Capogna-Cowling}]\label{Capogna-Cowling} All
$1$-quasiconformal maps between Carnot groups are smooth.
\end{theorem}

Furthermore, by \cite[Corollary 7.2]{Capogna-Cowling} we also have the following Lemma.

\begin{lemma}\label{dfsmooth}
If $f$ is a quasiconformal embedding such that $df(p)=\delta_{\lambda(p)}$ for almost all points $p$ in its domain of definition, then $f$ is $1$-quasiconformal and therefore smooth.  
\end{lemma}

\subsection{Quasiconformal equivalence}  
%The following theorem is a consequence of Pansu differentiability of quasiconformal maps between Carnot groups.
 
 The Pansu differential of a quasiconformal map is a graded group isomorphism. Consequently, we have the following fact.
\begin{theorem} [P. Pansu, \cite{pansu1}]  \label{ThmPansu}
Two Carnot groups are quasiconformally equivalent if and only if they are isomorphic as groups.
\end{theorem}
In particular, when  $f$ is an embedding of an open set $U\subset \G$ into $\G$ itself,  we have that $df(p)\in {\rm Aut}_0(\g)$, for almost every $p\in U$. Furthermore, the set of smooth contact maps of $\G$ into itself coincides with the set of smooth Pansu differentiable maps  of $\G$ to itself, see \cite{Warhurst-2}. 

Theorem~\ref{ThmPansu} was later generalised by Margulis and Mostow.
\begin{theorem} [G. Margulis, G. Mostow, \cite{Margulis-Mostow}, \cite{Margulis-Mostow2}]  \label{ThmMargMost}
Let $(M,\mathcal{H},{\rm g})$ and $(M',\mathcal{H}',{\rm g}')$ be equiregular sub-Riemannian  manifolds. 
Any quasiconformal embedding  $ f :U \subseteq M \to M'$ of an open set  $U$ of $M$     induces, at almost every point $p\in U$, an isomorphism
$$Df(p):\G_p \to \G'_{f(p)},$$
between the tangent cones of $M$ and $M'$. 
 %If $M$ and $M'$ are quasi-conformally equivalent, then  the tangent cones $\G_p$ and $\G'_{f(p)}$  are isomorphic.
\end{theorem}

%One can then ask: If two strongly regular sub-Riemannian manifolds have the same symbol, are they a quasi-conformally equivalent ? We demonstrate that in general the answer is in the negative. 

One might wonder if the converse of the Margulis-Mostow Theorem holds true.
Alas, we show that it is not the case. Indeed, we exhibit two sub-Riemannian manifolds that at every point have the same fixed Carnot group as tangent cone. Then we prove that they are not quasiconformally equivalent.
In fact, we find such examples among the class of nilpotent Lie groups.

\section{Ultrarigid groups}\label{ultra}
\subsection{Definition of ultrarigidity} 
In this section we present a class of groups that we shall consider in proving the main theorem.
Such groups have the property of having very few quasiconformal maps.
Notice that left translations and group-dilations are always present.
In fact, we are interested in the case when these are the only quasiconformal maps.
  Theorem \ref{ultra-rigidity} gives an algebraic characterization of such a situation and will be proved in Section \ref{section-proof-equivalence}. %For convenience of the reader, we restate the theorem.
%\begin{theorem}[Restatement of Theorem \ref{ultra-rigidity}] %\label{ultra-rigidity}\label{ultralip}
%Let  $\G$ be a   Carnot group. Then the following are equivalent:
%\begin{itemize}
%\item[{[\ref{ultra-rigidity}.1]}]  
%If $U \subset \G$ is open and $f : U \to \G$ is a quasiconformal
%embedding, 
% For any $U\subset \G $ open and any quasiconformal embedding $f: U\to \G $, 
% then $f$ is the restriction of  an element of TD$(\G)$; %the composition of a  left translation  and a group dilation;
%\item[{[\ref{ultra-rigidity}.2]}]  The group ${\rm Aut}_0(\g)$ coincides with the algebra-dilation group.
%\end{itemize}
%\end{theorem}

%We postpone the proof to Section \ref{}.
%The above equivalence leads to the following notion.
%The theorem above leads us to the definition of ultra-rigidity.
%Because of Theorem \ref{ultra-rigid}, there are two equivalent definition for the class of ultra-rigid groups.
 Because of Theorem \ref{ultra-rigidity},  the notion of ultrarigid group may be defined in two equivalent manners.

%We begin with the definition of Ultra-rigid groups, these are the groups which admit the smallest possible pseudo-groups of contact symmetries.
 
\begin{definition}
A Carnot group $\G$ is said to be {\em ultrarigid}  if  one of the two equivalent properties [\ref{ultra-rigidity}.1] and [\ref{ultra-rigidity}.2] of Theorem \ref{ultra-rigidity} holds. 
\end{definition}

\begin{remark}
Property  [\ref{ultra-rigidity}.2] has been considered by Pierre Pansu. He proved that  there exist uncountably many groups with such property, see \cite[Proposition 13.1]{pansu1}. This existence result of Pansu does not yield an example which suits our purpose.
It is important to recall that Pansu mainly considered more general groups.
Namely, we call {\em Pansu-rigid} those Carnot groups $\G$ for which any strata-preserving automorphism of Lie$(\G)$ is a similarity, i.e., the composition of a dilation and an isometry.
One of the main steps in proving Mostow rigidity for quaternionic hyperbolic spaces is to show that the boundary at infinity of such spaces are Pansu-rigid, see \cite[Proposition 10.1]{pansu1}. Clearly, any ultrarigid group is Pansu-rigid.
\end{remark}

\begin{remark}\label{capcowstep2}
For Carnot groups of step $2$, the nontrivial part of Theorem \ref{ultra-rigidity} has been proved by Capogna and Cowling, see \cite[Corollary 7.4]{Capogna-Cowling}.
\end{remark}

\subsection{Examples of ultrarigid groups}
In this section we present two examples of ultrarigid groups. The ultrarigidity  can be verified by explicit computation of the strata preserving automorphisms using the MAPLE LieAlgebras package. 

In Section \ref{non-Carnot} we shall need an example of an ultrarigid group whose structure can be deformed to a nonstratified nilpotent Lie group. The following Lie algebra determines an ultrarigid group  having this flexibility.

\begin{example}\label{ex1} Consider the sixteen dimensional Lie algebra with basis $\{e_i \, | \, i=1,\ldots,16\}$ and bracket relations:
\begin{align*}
 [e_1, e_2] &= e_{11},  &[e_1, e_3] &= e_{13},  &[e_1, e_4] &= e_{14}, \\
 [e_1, e_5] &= e_{15},  &[e_1, e_6] &= e_{16},  &[e_2, e_3] &= e_{13}, \\
 [e_2, e_5] &= e_{12},  &[e_2, e_6] &= e_{14},  &[e_3, e_5] &= e_{12}, \\
 [e_3, e_6] &= e_{13},  &[e_3, e_7] &= e_{14},  &[e_4, e_5] &= e_{12}, \\
 [e_4, e_6] &= e_{13},  &[e_4, e_8] &= e_{14},  &[e_5, e_6] &= e_{13}, \\
 [e_5, e_8] &= e_{12},  &[e_5, e_9] &= e_{14},  &[e_6, e_8] &= e_{12}, \\
 [e_6, e_9] &= e_{13},  &[e_6, e_{10}] &= e_{14}, &[e_7, e_8] &= e_{14}, \\
 [e_7, e_9] &= e_{12},  &[e_7, e_{10}] &= e_{13}, &[e_8, e_9] &= e_{13}, \\
 [e_8, e_{10}] &= e_{14},  &[e_9,e_{10}] &=- e_{12}, 
\end{align*} 
We note that this is a $2$-step Carnot algebra with stratification $V_{1}={\rm span}\, \{e_1,\dots,e_{10}\}$ and $V_{2}={\rm span}\, \{e_{11},\dots,e_{16}\}$. It can be deformed to a nonstratified nilpotent Lie algebra by adding the additional bracket $[e_1,e_{11}]=e_{14}$.
 \end{example}

\begin{example} Our second example is the seventeen dimensional group corresponding to the Lie algebra obtained by extending the previous example by adding the additional bracket $[e_1,e_{11}]=e_{17}$.
The strata are as follows: $V_{1}={\rm span}\, \{e_1 ,\dots, e_{10}\}$, $V_{2}={\rm span}\, \{e_{11},\dots, e_{16}\}$ and  $V_{3}={\rm span}\, \{e_{17}\}$. The significance of this example is that it shows that Theorem \ref{ultra-rigidity} does extend the result discussed in Remark \ref{capcowstep2}. 
\end{example}

%{\color{red} SHALL WE PUT MORE EXAMPLES? IF YES, CHANGE THE TITLE OF SECTION	 example  $\mapsto $ exampleS}

\section{A counterexample}\label{counter}

\subsection{A criterion for quasiconformal nonequivalence}
In this section we show that sub-Riemannian nilpotent groups with same tangent cones need not to be quasiconformally equivalent.
The main theorem of the section is Theorem \ref{main2}, where ultrarigidity is assumed.
Namely, we deal with a Carnot group $\G$ whose only quasiconformal maps are the elements of  TD$(\G)$.

We start by showing that, in general, if a nilpotent subgroup $H$ of TD$(\G)$ has codimension one, then $H$ is in fact $\G$. To this end, let us study the group structure of TD$(\G)$. Composition of functions turns TD$(\G)$ into a Lie group that is isomorphic to a semidirect product $\G \rtimes \R$. 
 %This fact is crucial to the proof of Theorem~\ref{main2}.
Thus the Lie algebra of TD$(\G)$ is a semidirect product of $\g$ with a one dimensional subgroup:
$$ \Lie ({\rm TD}(\G)) = \g \rtimes \R.$$
Here the $\R$-factor is generated  by the derivation $D$ and the brackets in Lie$ ({\rm TD}(\G)) $ are those of $\g$  together with 
$$[D,X] = D(X), \qquad \forall X\in \g.$$

\begin{lemma}\label{HisG}
Let $\G\neq \R$ be a  Carnot group. Let $H< {\rm TD}(\G) $ be a Lie subgroup of codimension $1$ and
assume that $H$ is nilpotent. Then $H=\G\times \{0\}$.
\end{lemma}

\begin{proof}
Denote by $\mathfrak{h}$, $\g$, and $\g\rtimes \R$ the Lie algebras of $H$, $G$, and  TD$(\G)$, respectively.
Thus we have $\dim \mathfrak{h}=\dim \mathfrak{g}=n$ and $\dim(\g\rtimes \R)=n+1$.
Let $V_{1}$ be the first layer of $\g$, for which we recall that $\dim V_{1}\geq 2$.
Hence we get
$$\dim(V_{1}\cap \mathfrak{h}) =\dim V_{1}+\dim \mathfrak{h} - \dim (V_{1}+ \mathfrak{h})$$
$$\geq 2+n-(n+1)=1.$$
Thus there exists $X\in V_{1}\cap \mathfrak{h}$ with $X\neq0$.

We note that if $V_{1}\subseteq  \mathfrak{h}$, then $\g\subseteq  \mathfrak{h}$ and so $\G=H$.  Now consider the case  $V_{1}\setminus  (V_{1} \cap \mathfrak{h}) \neq\emptyset$ and let $Y$ be a nonzero element in $V_{1} \setminus (V_{1} \cap \mathfrak{h})$. Then, since $ \mathfrak{h}$ has codimension $1$ in $\g\rtimes \R$, we have that
$\g\rtimes \R={\rm span}\{    \mathfrak{h}, Y\}$.
Thus there exist $Z\in   \mathfrak{h}$ and $\alpha\in\R$ such that $D=Z+\alpha Y$.
Notice that $Z\neq 0$, otherwise $D=\alpha Y\in V_{1}$, which is not true.
Therefore, the element $D-\alpha Y$ is in $ \mathfrak{h}\setminus \{0\}$, and since $D$ preserves strata, there exists $Z_m \in V_{2}\oplus \cdots\oplus V_{s}$ such that
$$\left({\rm ad}_{D-\alpha Y}\right)^m(X)=X+Z_m, \qquad \forall m\in \N.$$
However, since $D-\alpha Y$ and $X$ are both elements of the nilpotent algebra $\mathfrak{h}$, the iterated bracket $\left({\rm ad}_{D-\alpha Y}\right)^m(X)$ should be eventually $0$, which contradicts the fact that $X\ne0$. Thus we conclude that
   $V_{1}\setminus  (V_{1} \cap \mathfrak{h}) \neq\emptyset$ cannot occur.
\end{proof}

Recall that, unless otherwise said, a Carnot group is always equipped with a  left invariant  sub-Riemannian distance with respect to the first stratum as horizontal distribution. 

%Let $\Delta$ be a bracket-generating  left invariant tangent sub-bundle on a group $H$.
% We denote by $d_\Delta$ any sub-Riemannian distance on $H$ with $\Delta$ as horizontal  distribution. 
% Notice that, if  $d'_\Delta$ is another such  distance, then $(H,d'_\Delta)$ and  $(H,d_\Delta)$ are locally   biLipschitz equivalent.
% Hence the (local) quasi-conformal class of $(H,d_\Delta)$ is well defined.

\begin{theorem}\label{main2}
Assume $\G$ is an ultrarigid Carnot group.
Let $H$ be a connected, simply connected nilpotent Lie group endowed with a left invariant  sub-Riemannian distance.
If there exist open sets $U\subset \G$, $U'\subset H$, and a quasiconformal homeomorphism between $U$ and $U'$, then  $H$ is isomorphic to $\G$.
%Assume that   $\g$ is the stratification (also called nilpotentisation) of $\mathfrak{h}$, w.r.t. some subspace $\Delta\subset \mathfrak{h}$, but
%that $\mathfrak{h}$ is not isomorphic to $\g$.
%Then $(H,d_\Delta)$ is not locally quasi-conformally equivalent to $(\G,d_{V_{1}})$.
\end{theorem}

%The following consequence is immediate.
%
%MYBE IT SHOULD BE REMOVED?
%\begin{corollary}\label{main2}
%Let $\G$ be an ultra-rigid Carnot group.
%Let $H$ be a connected, simply connected nilpotent group  endowed with a left invariant  sub-Riemannian distance.
%Assume that    that $H$ is not isomorphic to $\G$.
%Then $H$ is not locally quasi-conformally equivalent to $\G$.
%\end{corollary}

\proof  %\proof[Proof of Theorem \ref{pre-main2}]
%By the way of contradiction, assume that there exist open sets $U\subseteq \G$ and $V\subseteq H$ and a quasi-conformal homeomorphism $f:(U,d_{V_{1}})\to (V,d_\Delta)$.
 Let $f:U\to U'$ be the quasiconformal homeomorphism.
Composing with a suitable translation, we may assume that $f(e_\G)=e_H$.
Then, for $h \in H$, we consider the quasiconformal map
$$f_h :=f^{-1}\circ \tau_h\circ f :\tilde U\subseteq \G \to \G,$$
which is well defined for $h$ close enough to the identity in $H$ and for some open set $\tilde U\subseteq \G$.
By definition of ultrarigidity, we can assume $f_h$ is in fact in ${\rm TD}(\G)$.

We claim that the map $h\mapsto f_h$ is an injective local homomorphism of $H$ into TD$(\G)$. Therefore, going to the Lie algebra level, we get a (globally defined) injective local homomorphism of $\mathfrak {h}$ into Lie(TD$(\G)$). Indeed, the map is a homomorphism, because
$$f_h\circ f_{h'} =f^{-1}\circ \tau_h  \circ \tau_{h'}\circ f =f^{-1}\circ \tau_{h h'}\circ f =f_{h h'}, \qquad\forall h,h'\in H.$$
Regarding the injectivity, for $h\neq e_H$, we   show that $f_h\neq {\rm id}$. Since $f(e_G)=e_H$,
$$f_h(e_G)=(f^{-1}\circ \tau_h\circ f ) (e_G)=(f^{-1}\circ \tau_h)(e_H)
=f^{-1}(h)\neq e_G.$$
Therefore $\mathfrak {h}$ is   isomorphic to a subalgebra $\mathfrak {h}_0< {\rm Lie (TD} (\G))$.
Since $\mathfrak {h}_0$ is nilpotent, by Lemma \ref{HisG},
we have that $\mathfrak {h}_0=\g\times \{0\}$. Thus $\mathfrak {h}$ is   isomorphic to $\g$ and therefore $H$ is isomorphic to $\G$, since they are  connected, simply connected nilpotent Lie groups.  %The contradiction with the fact that $\mathfrak {h}$ is not isomorphic to $\g$ implies that the above quasi-conformal map cannot exist.
\qed

\subsection{Example of a non-Carnot group with ultrarigid tangent}\label{non-Carnot}
We present here an example of a sub-Riemannian nilpotent Lie group  that demonstrates the validity of Theorem \ref{main}. Namely, we exhibit a nilpotent Lie group $H$ whose tangent cone is the $16$-dimensional group $\G$ as in Example \ref{ex1} such that the pair $\G$, $H$  satisfy the condition of Theorem  \ref{main2}. In turn, this implies Theorem \ref{main} and Corollary \ref{corollario}.

The nilpotent group is the following.
Take $\G=\exp(\g)$ where $\g$ is Example~\ref{ex1}, and let $H=\exp(\mathfrak{h})$ where $\mathfrak{h}$ is the $16$-dimensional nilpotent Lie algebra with the same bracket relations as $\g$ and the additional bracket $[e_1,e_{11}]=e_{14}$. Note that this additional bracket is of order $3$ and so $\mathfrak{h}$ is not stratified. 

If $ X_i$ denotes the left invariant vector field corresponding to $e_i$, then the horizontal space $\mathcal{H} \subset TH$ is framed by $ X_1, \dots,  X_{10}$. For a given point $p$,  $L_1$ is the set of smooth sections of $\mathcal{H}$ defined on a neighbourhood of $p$, and $L_2=L_1+[L_1,L_1]$. It follows that $ X_1 \in L_1$, $ X_{11}\in L_2$ and $ X_{14}\in L_2$, hence $$[ X_1,  X_{11}]+L_2= X_{14}+L_2=0+L_2.$$ On the other hand, if $ X,  Y \in L_1$ and $[ X,  Y]=0+L_1$, then $[ X,  Y]=0$ and so $\g(p)=\g$ for all $p \in H$.

\section{Equivalence of definitions for ultrarigid groups}
\label{section-proof-equivalence}

%{\color{red} TO BE REVISED}

In this section we prove  Theorem \ref{ultralip}. Part of our proof uses a theorem of Tanaka, that provides a characterization of the space of contact maps on a Carnot group $\G$ at the infinitesimal level. In order to state Tanaka's theorem, it is convenient to change part of the notation. Throughout this section  we shall denote  by $\g_{-i}$ the strata $V_i$ of a nilpotent and stratified Lie algebra $\g$, for every $i=1,\dots,s$.

\subsection{Tanaka Prolongation}
The {\it Tanaka prolongation} of $\g$ is the graded Lie algebra ${\rm Prol}(\g)$ given by the direct sum $${\rm Prol}(\g):= \bigoplus_{k \in \Z} \g_k, $$ where $ \g_k=\{0\}$ for $k < -s$ and for each $k \geq 0$, $\g_k$ is inductively defined by
$$\g_k :=\Big \{  u \in \bigoplus_{\ell<0} \g_{\ell+k} \otimes \g_\ell^* \ | \ u([X,Y])=[u(X),Y]+[X,u(Y)] \Big \}.$$
Clearly, for $k=0$ we get the strata preserving derivations.
 If  $u \in \g_k$, where $k \geq 0$, then the condition in the definition becomes the Jacobi identity upon setting $[u,X]=u(X)$ when $X \in \g$. Furthermore, if $u \in \g_k$ and $u^\prime \in \g_\ell$, where $k,\ell \geq 0$, then $[u,u^\prime] \in \g_{k+\ell}$ is defined inductively according to the Jacobi identity, that is
$$ [u,u^\prime](X)=[u,[u^\prime,X]]-[u^\prime,[u,X]].$$
In \cite{Tanaka}, Tanaka shows that ${\rm Prol}(\g)$ determines the structure of the contact vector fields on the group $G$. A {\em contact vector field} is defined as the infinitesimal generator of a local flow of contact maps, and the space of these vector fields forms a Lie algebra with the usual bracket of vector fields. We recall that $D$ denotes
the standard dilation defined in Section~\ref{CarnotGroups}, and we rephrase the result of Tanaka with the following statement.
\begin{theorem}[N. Tanaka, \cite{Tanaka}]\label{tanaka}
 Let $U\subset\G$ be an open set. Denote by $\mathcal{C}(U)$ the Lie algebra of smooth contact vector fields on $U$. If ${\rm Prol}(\g)$ is finite dimensional, then there exists a Lie algebra isomorphism between ${\rm Prol}(\g)$ and $\mathcal{C}(U)$. In particular, if ${\rm Prol}(\g) =\g\oplus\g_0$
  and $\g_0=\R D$, one may choose this 
  isomorphism to be the linear map $\rho$ defined by the assignments
\begin{equation}\label{rhoD}
 \rho(D)\phi(p) = \frac{d}{dt} \phi (\exp(e^{-tD}\exp^{-1}(p)))|_{t=0},
\end{equation}
\begin{equation}\label{rhoX}
 \rho(X)\phi(p) = \frac{d}{dt} \phi (\exp(-tX)p)|_{t=0},
\end{equation}
where $p\in U$, $\phi$ is a 
smooth function on $U$ and $X$ varies in $\g$.
\end{theorem}
The interested reader can consult~\cite{Tanaka, Yamaguchi} for a thorough overview, and~\cite{OW} for a basic introduction.
\begin{remark}\label{tanisorem}
Since $\G$ is simply connected, then in the case ${\rm Prol}(\g)$ is finite dimensional, every $V \in \mathcal{C}(U)$ uniquely extends to an element of $\mathcal{C}(\G)$, see~\cite[page 34]{Tanaka}.
%{\color{red} reference?} 
\end{remark}
\begin{remark}
 We notice that ${\rm Prol}(\g)$ is finite dimensional if and only if $\g_k=\{0\}$ for some $k\geq 0$. In fact, we have that $\g_k=\{0\}$ implies $\g_{k+l}=\{0\}$
for every $l\geq 0$.
\end{remark}
We note that if ${\rm Aut}_0(\g)$ consists only of dilations, then $\g_0$ is exactly the span of $D$ and thus one dimensional. The following lemma implies that in this case the only contact flows are dilations. We would like to thank M. Reimann for bringing this fact to our attention.
\begin{lemma}  \label{prodim1}
Let ${\g}$ be a nonabelian nilpotent and stratified Lie algebra such that $\g_0$ is one dimensional. Then the prolongation of $\g$ is $\g \oplus \g_0$.
\end{lemma}
\begin{proof}
We need to show that $\g_1=\{0\}$. Set $u\in\g_1$. Then $u(\g_j)\subset \g_{j+1}$ for every $j=-s,\dots,-1$ and $u([X,Y])=[u(X),Y]+[X,u(Y)]$ for all $X,Y\in\g$.
Since $\g_0$ has dimension one, then $\g_0={\rm span}\{D\}$.
In particular, if $X\in\g_{-1}$, then $u(X)=c(X) D$, where $c:\g_{-1} \to  \R$ is linear. In order to show that $u=0$, it is enough to prove that $ c(X)=0 $ for all $X\in\g_{-1}$.

Let $\z(\g)$ denote the centre of $\g$, and let $Z \in \z(\g)\cap \g_{-k}\ne \emptyset$. Then for all $X \in \g_{-1}$
\begin{align*}
 0=u([X,Z]) &= [u(X),Z] +[X,u(Z)]\\
&= -c(X) k Z + [X,u(Z)],
\end{align*}
whence 
$$
[X,u(Z)]=c(X) k Z.
$$
By induction, it is easy to show that
$$
 [X,u^j (Z)]= c(X) \sum_{l=0}^{j-1}(k-l) u^{j-1}(Z)=c(X) \frac{j(2k-j+1)}{2} u^{j-1}(Z),
$$
which in the case $j=k$ gives
\begin{equation}\label{firsteq}
[X,u^{k}(Z)]= c(X) \frac{k(k+1)}{2}u^{k-1}(Z).
\end{equation}
Furthermore, iterating (\ref{firsteq}) gives 
\begin{equation}
{\rm ad}_X^ku^k(Z)=A_kc(X)^kZ, \label{adxkz}
\end{equation}
where $A_k$ is a positive constant depending on $k$ only.
Notice that since $u^{k}(Z)\in\g_0$, we have $u^{k-1}(Z)\in\g_{-1}$, and it follows that
$$
[X,u^{k}(Z)]=[X, u( u^{k-1}(Z))]= -c(u^{k-1}(Z))X,
$$ for all $X\in\g_{-1}$. We conclude that for $k \geq 2$, the left hand side of (\ref{adxkz}) is zero. It follows that since $\g$ is nonabelian, then $c(X)=0$ since we can set $k=s\geq 2$, and choose a nonzero $Z \in \z(\g)\cap \g_{-s}$. 
\end{proof}

\subsection{Proof of Theorem~\ref{ultralip}}\label{proofultralip}

${[\ref{ultra-rigidity}.1]} \Rightarrow {[\ref{ultra-rigidity}.2]}$. Every element  $\alpha \in {\rm Aut}_0 (\g)$ lifts to an automorphism of $\G$ which is also a contact map. Therefore by hypothesis this contact map is an element of TD$(\G)$, and since it is an automorphism it must be  a group-dilation. We conclude that $\alpha$ is an algebra-dilation.

${[\ref{ultra-rigidity}.2]} \Rightarrow {[\ref{ultra-rigidity}.1]}$. Let $f:U\rightarrow \G$ be a quasiconformal embedding. Then $f$ is Pansu differentiable at almost every $p\in U$. Therefore  $df(p)=\delta_{\lambda(p)}$  for almost every $p\in U$ and by Lemma \ref{dfsmooth}, $f$ is $1$-quasiconformal and smooth. In particular $f$ is a smooth contact map.

After normalising with left translations if necessary, we can assume that $e\in U$ and $f(e)=e$. Moreover $f_* W\in \mathcal{C}(f(U))$ for every $W\in\mathcal{C}(U)$. By Remark~\ref{tanisorem}, $f_*$ induces a Lie algebra isomorphism of $\mathcal{C}(U)$,  which we  also denote by $f_*$.
It then follows that $\rho^{-1}f_*\rho$ is an automorphism of $\g\oplus\g_0$. This automorphism has some extra properties that we show in the following lemma.
\begin{lemma}\label{autorho}
The automorphism $\alpha := \rho^{-1}f_*\rho$  preserves $\g$ and $\g_0$. Moreover, $\alpha|_{\g}\in {\rm Aut}_0(\g)$.
\end{lemma}
\begin{proof}
 By \eqref{rhoD}  we see that $\rho(D)(e)=0$, whereas by \eqref{rhoX} we have that $\rho(X)(e)\neq 0$ for every $X\in\g$. Since $f(e)=e$, we conclude that $\g_0$ is preserved.
Since $D$ is surjective, it follows that $[\g\oplus\g_0,\g\oplus\g_0]=\g$ which implies $\alpha(\g)= \g$.

In order to show that $\alpha|_{\g}$ preserves the strata, it is enough to prove that it preserves $\g_{-1}$. Since $f$ is contact, this is true if the  equation 
\begin{equation}\label{commutationrule}
f_*\rho(X) = \rho (f_*|_{e}X),
 \end{equation}
holds for every $X\in\g_{-1}$. To show \eqref{commutationrule}, we first observe that the flow of $f_*\rho(X)$ through $p$ is $f_t(p)=f(\exp(-tX)f^{-1}(p))$ and in particular $f_t(e)=f(\exp(-tX))$. Since $f(e)=e$, we have that $f_*|_{e}=df(e)$ on $\g_{-1}$. By hypothesis, $df(e)=\delta_\lambda$ for some $\lambda \in\R\setminus\{0\}$, and it follows that
$$
\frac{d}{dt}f_t(e)=-f_*|_{e}(X)=-\lambda X.
$$
Hence \eqref{commutationrule} is valid when evaluated at the identity. The equality at all points follows from the fact that both $f_*\rho(X)$
and $\rho (f_*|_{e}X)$ are right invariant vector fields, see~\eqref{rhoX}.
\end{proof}

We now conclude the proof of Theorem \ref{ultralip}.
Since $\rho^{-1}f_*\rho|_{\g}\in{\rm Aut}_0(\g)$, it follows that $\rho^{-1}f_*\rho|_{\g} =\delta_{\lambda}$ for some $\lambda\in\R\setminus\{0\}$. If $F(p)=\delta_{1/\lambda} \circ f$, then $F(e)=e$, and by \eqref{commutationrule}, we see that $\rho^{-1}F_*\rho|_{\g} =I$. In particular $F_*\rho(X)=\rho(X)$ for every $X\in\g$.
Thus, $F_*$ preserves each right invariant vector field, and so $F$ commutes with the left translations. Hence $F(p q)=pF(q)$, and putting $q=e$ shows that $F$ is the identity. Therefore $f=\delta_\lambda$ and the proof of Theorem~\ref{ultralip} is complete. \qed


\vskip0.2cm



\begin{thebibliography}{10}

\bibitem[CC06]{Capogna-Cowling}
Luca Capogna and Michael Cowling, \emph{Conformality and {$Q$}-harmonicity in
  {C}arnot groups}, Duke Math. J. \textbf{135} (2006), no.~3, 455--479.

%\bibitem[DFN]{DFN}
%B. A. Dubrovin, A. T. Fomenko, S. P. Novikov, Modern geometry - methods and applications,
%Part I, Springer, New York, (1984).

\bibitem[Hei95]{Heinonen-calculus}
Juha Heinonen, \emph{Calculus on {C}arnot groups}, Fall {S}chool in {A}nalysis
  ({J}yv\"askyl\"a, 1994), Report, vol.~68, Univ. Jyv\"askyl\"a, Jyv\"askyl\"a,
  1995, pp.~1--31.


\bibitem[MM95]{Margulis-Mostow}
Gregori~A. Margulis and George~D. Mostow, \emph{The differential of a
  quasi-conformal mapping of a {C}arnot-{C}arath\'eodory space}, Geom. Funct.
  Anal. \textbf{5} (1995), no.~2, 402--433.

\bibitem[MM00]{Margulis-Mostow2}
Gregori~A. Margulis and George~D. Mostow, \emph{Some remarks on the definition of
  tangent cones in a {C}arnot-{C}arath\'eodory space}, J. Anal. Math.
  \textbf{80} (2000), 299--317.


\bibitem[Mit85]{Mitchell}
John Mitchell, \emph{On {C}arnot-{C}arath\'eodory metrics}, J. Differential
  Geom. \textbf{21} (1985), no.~1, 35--45.

\bibitem[OW10]{OW}
Alessandro Ottazzi and Ben Warhurst, \emph{Contact and $1$-quasiconformal maps on Carnot groups}, J. Lie Theory.   \textbf{21} (2011), 787--811.


 
\bibitem[Pan89]{pansu1}
Pierre Pansu, \emph{M\'etriques de {C}arnot-{C}arath\'eodory et
  quasiisom\'etries des espaces sym\'etriques de rang un}, Ann. of Math. (2)
  \textbf{129} (1989), no.~1, 1--60.

\bibitem[Pan83]{Pansu-croissance}
Pierre Pansu, \emph{Croissance des boules et des g\'eod\'esiques ferm\'ees dans
  les nilvari\'et\'es}, Ergodic Theory Dynam. Systems \textbf{3} (1983), no.~3,
  415--445.
 
\bibitem[Sha04]{Shalom}
Yehuda Shalom, \emph{Harmonic analysis, cohomology, and the large-scale
  geometry of amenable groups}, Acta Math. \textbf{192} (2004), no.~2,
  119--185.
 
   
\bibitem[Tan70]{Tanaka}
Noboru Tanaka, \emph{On differential systems, graded {L}ie algebras and
  pseudogroups}, J. Math. Kyoto Univ. \textbf{10} (1970), 1--82.

\bibitem[Var81]{Varchenko}
Alexander~N. Var{\v{c}}enko, \emph{Obstructions to local equivalence of
  distributions}, Mat. Zametki \textbf{29} (1981), no.~6, 939--947, 957.

\bibitem[War08]{Warhurst-2}
Ben Warhurst, \emph{Contact and {P}ansu differentiable maps on {C}arnot
  groups}, Bull. Aust. Math. Soc. \textbf{77} (2008), no.~3, 495--507.

\bibitem[Yam93]{Yamaguchi}
Keizo Yamaguchi, \emph{Differential systems associated with simple graded {L}ie
  algebras}, Progress in differential geometry, Adv. Stud. Pure Math., vol.~22,
  Math. Soc. Japan, Tokyo, 1993, pp.~413--494.

\end{thebibliography}
\end{document}